\documentclass[12pt,twoside,reqno]{amsart}

\usepackage[T2A]{fontenc}
\usepackage{amssymb}
\usepackage{amsmath}
\usepackage{mathrsfs}
\usepackage{graphicx}

\def\H{{\mathcal H}}
\def\R{{\mathbb R}}

\def\d{\delta}
\def\D{\Delta}
\def\t{\theta}

\def\e{\varepsilon}
\def\f{\varphi}
\def\s{\sigma}

\newcommand{\supp}{\operatorname{supp}}

\newtheorem{theorem}{Theorem}[section]
\newtheorem{lemma}[theorem]{Lemma}

\newtheorem{proposition}[theorem]{Proposition}

\newtheorem{conjecture}[theorem]{Conjecture}

\numberwithin{equation}{section}

\begin{document}
\begin{center}
  \emph{Dedicated to a memory of remarkable mathematician and man Victor Petrovich Havin}
\end{center}

\vspace{.4cm}
\title {On the maximum principle for the Riesz transform}

\author{Vladimir Eiderman and Fedor Nazarov}
\address{Vladimir Eiderman, Department of  Mathematics, Indiana University, Bloomington, IN}
\email{veiderma@indiana.edu}
\address{Fedor Nazarov,  Department of  Mathematics, Kent State University, Kent, OH}
 \email{nazarov@math.kent.edu}
\begin{abstract}
Let $\mu$ be a measure in $\mathbb R^d$ with compact support and continuous density, and let
$$
R^s\mu(x)=\int\frac{y-x}{|y-x|^{s+1}}\,d\mu(y),\ \ x,y\in\mathbb R^d,\ \ 0<s<d.
$$
We consider the following conjecture:
$$
\sup_{x\in\mathbb R^d}|R^s\mu(x)|\le C\sup_{x\in\text{supp}\,\mu}|R^s\mu(x)|,\quad C=C(d,s).
$$
This relation was known for $d-1\le s<d$, and is still an open problem in the general case. We prove the maximum principle for $0< s<1$, and also for $0<s<d$ in the case of radial measure. Moreover, we show that this conjecture is incorrect for non-positive measures.
\end{abstract}
\maketitle

\section{Introduction}\label{introd}

Let $\mu$ be a non-negative finite Borel measure with compact support in $\mathbb R^d$, and let $0<s<d$. The truncated Riesz operator $R_{\mu,\e}^s$ is defined by the equality
$$
R_{\mu,\e}^sf(x)=\int_{|y-x|>\e}\frac{y-x}{|y-x|^{s+1}}f(y)\,d\mu(y),\ \ x,y\in\mathbb R^d,\ f\in L^2(\mu),\ \e>0.
$$
For every $\e>0$ the operator $R_{\mu,\e}^s$ is bounded on $L^2(\mu)$. By $R_{\mu}^s$ we denote a linear operator on $L^2(\mu)$ such that
$$
R_{\mu}^sf(x)=\int\frac{y-x}{|y-x|^{s+1}}f(y)\,d\mu(y),
$$
whenever the integral exists in the sense of the principal value. We say that $R_{\mu}^s$ is bounded on $L^2(\mu)$ if
$$
\|R_{\mu}^s\|:=\sup_{\e>0}\|R_{\mu,\e}^s\|_{L^2(\mu)\to L^2(\mu)}<\infty.
$$
In the case $f\equiv1$ the function $R_{\mu}^s1(x)$ is said to be {\it the $s$-Riesz transform (potential) of $\mu$} and is denoted by $R^s\mu(x)$.
If $\mu$ has continuous density with respect to the Lebesgue measure $m_d$ in $\mathbb R^d$, that is if $d\mu(x)=\rho(x)\,dm_d(x)$
with $\rho(x)\in C(\mathbb R^d)$,
then $R^s\mu(x)$ exists for every $x\in\R^d$.

By $C,c$, possibly with indexes, we denote various constants which may depend only on $d$ and $s$.

We consider the following well-known conjecture.

\begin{conjecture}\label{conj11} Let $\mu$ be a nonnegative finite Borel measure with compact support and continuous density with respect to the Lebesgue measure in $\mathbb R^d$.
There is a constant $C$ such that
\begin{equation}\label{f2}
\sup_{x\in\R^d}|R^s\mu(x)|\le C\sup_{x\in\supp\mu}|R^s\mu(x)|.
\end{equation}
\end{conjecture}

For $s=d-1$ the proof is simple. Obviously,
\begin{equation}\label{f3}
R^s\mu(x)=\nabla U_\mu^s(x),
\end{equation}
where
$$
U_\mu^s(x)=\frac1{s-1}\int\frac{d\mu(y)}{|y-x|^{s-1}},\ s\ne1,\ \ U_\mu^1(x)=-\int\log|y-x|\,d\mu(y).
$$
Thus each component of the vector function $R^s\mu(x),\ s=d-1$, is harmonic in $\R^d\setminus\supp\mu$. Applying the maximum principle for harmonic functions we get \eqref{f2}.

For $d-1<s<d$, the relation \eqref{f2} was established in \cite{ENV} under stronger assumption that $\rho\in C^\infty(\R^d)$. In fact it was proved that \eqref{f2} holds for each component of $R^s\mu$ with $C=1$ as in the case $s=d-1$. The proof is based on the formula which recovers a density $\rho$ from $U_\mu^s$. But this method does not work for $s<d-1$.

The problem under consideration has a very strong motivation and also is of independent interest. In \cite{ENV} it is an important ingredient of the proof of the following theorem. By $\H^s$ we denote the $s$-dimensional Hausdorff measure.

\begin{theorem}[\cite{ENV}]\label{th12} Let $d-1<s<d$, and let $\mu$ be a positive finite Borel measure such that $\H^s(\supp\mu)<\infty$. Then
$\|R^s\mu\|_{L^\infty(m_d)}=\infty$ (equivalently, $\|R^s_\mu\|=\infty$).
\end{theorem}

If $s$ is integer, the conclusion of Theorem \ref{th12} is incorrect. For $0<s<1$ Theorem \ref{th12} was proved by Prat \cite{P} using different approach. The obstacle for extension of this result to all noninteger $s$ between 1 and $d-1$ is the lack of the maximum principle. The same issue concerns the quantitative version of Theorem \ref{th12} obtained by Jaye, Nazarov, and Volberg \cite{JNV}.

The maximum principle is important for other problems on the connection between geometric properties of a measure and boundedness of the operator $R_{\mu}^s$ on $L^2(\mu)$ -- see for example \cite{JNV}, \cite{JNRT}, \cite{NToV1}, \cite{NToV2}. All these results are established for $d-1<s<d$ or $s=d-1$.

The problem of the lower estimate for $\|R^s_\mu\|$ in terms of the Wolff energy (a far going development of Theorem \ref{th12}) which is considered in \cite{JNV}, \cite{JNRT}, was known for $0<s<1$. And the results in \cite{NToV1}, \cite{NToV2} are $(d-1)$-dimensional analogs of classical facts known for $s=1$ (in particular, \cite{NToV2} contains the proof of the analog of the famous Vitushkin conjecture in higher dimensions). For $0<s\le1$, the proofs essentially use the Melnikov curvature techniques and do not require the maximum principle. But this tool is absent for $s>1$.

At the same time the validity of the maximum principle itself remained open even for $0<s<1$. It is especially interesting because the analog of \eqref{f2} does not hold for each component of $R^s\mu$ when $0<s<d-1$ unlike the case $d-1\le s<d$ -- see Proposition \ref{pr21} below.

We prove Conjecture \ref{conj11} for $0<s<1$ in Section \ref{sect2} (Theorem \ref{th23}). The proof is completely different from the proof in the case $d-1\le s<d$. In Section \ref{sect3} we prove Conjecture \ref{conj11} in the special case of radial density of $\mu$ (that is when $d\mu=h(|x|)\,dm_d(x)$), but for all $s\in(0,d)$. Section \ref{sect4} contains an example showing that Conjecture \ref{conj11} is incorrect for non-positive measures, even for radial measures with $C^\infty$-density (note that in \cite[Conjecture~7.3]{VE} Conjecture \ref{conj11} was formulated for all finite signed measures with compact support and $C^\infty$-density).

\section{The case $0<s<1$}\label{sect2}

We start with a statement showing that the maximum principle fails for every component of $R^s\mu$ if $0<s<d-1$.

\begin{proposition}\label{pr21} For any $d\ge2$, $0<s<d-1$, and any $M>0$, there is a positive measure $\mu$ in $\R^d$ with $C^\infty$-density such that
\begin{equation}\label{f21}
\sup_{x\in\R^d}|R_1^s\mu(x)|>M\sup_{x\in\supp\mu}|R_1^s\mu(x)|,
\end{equation}
where $R_1^s\mu$ is the first component of $R^s\mu$.
\end{proposition}

\begin{proof} Let $E=\{(x_1,\dots,x_d)\in\R^d:x_1=0,\ x_2^2+\dots+x_d^2\le1\}$, and let $E_\d$, $\d>0$, be a $\d$-neighborhood of $E$ in $\R^d$. Let $\mu=\mu_\d$ be a positive measure supported on $\overline{E_\d}$ with $\mu(\overline{E_\d})=1$ and with $C^\infty$-density $\rho(x)$ such that $\rho(x)<2/\text{vol}(E_\d)\le C_d/\d$. Then
$$
|R_1^s\mu(x')|>A_d,\ \text{where }x'=(1,0,\dots,0),\ \ 0<\d<1/2.
$$

On the other hand, for $x\in\supp\mu$ integration by parts yields
\begin{align*}
|R_1^s\mu(x)|&<\int_{|y-x|<\d}\frac1{|y-x|^s}\,d\mu(y)+\int_{|y-x|\ge\d}\frac\d{|y-x|^{s+1}}\,d\mu(y)\\
&=\frac{\mu(B(x,\d))}{\d^s}+s\int_0^\d\frac{\mu(B(x,r))}{r^{s+1}}\,dr+\d(s+1)\int_\d^\infty\frac{\mu(B(x,r))}{r^{s+2}}\,dr\\
&<C\frac{C_d}{\d}\frac{\d^d}{\d^s}+\frac{Cs}{\d}\int_0^\d\frac{r^d}{r^{s+1}}\,dr+
\frac{C\d(s+1)}{\d}\int_\d^2\frac{r^{d-1}\d}{r^{s+2}}\,dr+C\d.
\end{align*}
Here by $C$ we denote different constants depending only on $d$, and $B(x,r):=\{y\in\R^d:|y-x|<r\}$. We have
$$
\d\int_\d^2\frac{r^{d-1}}{r^{s+2}}\,dr=\begin{cases}\d\ln\dfrac2\d, &s=d-2,\\
\dfrac1{d-s-2}(2^{d-s-2}\d-\d^{d-s-1}), &s\ne d-2.
\end{cases}
$$
Thus, all terms in the right-hand side of the estimate for $|R_1^s\mu(x)|$ tend to 0 as $\d\to0$, and we may choose $\d$ and a corresponding measure $\mu$ satisfying \eqref{f21}.
\end{proof}

We need the following lemma. The notation $A\approx B$ means that $cA<B<CB$ with constants $c,C$ which may depend only on $d$ and $s$.

\begin{lemma}\label{le22} Let $\mu$ be a non-negative measure in $\R^d$ with continuous density and compact support. Let $0<s<d-1$. Then for every ball $B=B(x_0,r)$,
\begin{equation}\label{f22}
\bigg|\int_{\partial B}(R^s\mu\cdot\bold n)\,d\s\bigg|\approx r^{d-s-1}\mu(B)+
r^d\int_r^\infty\frac{d\mu(B(x_0,t))}{t^{s+1}},
\end{equation}
where $\bold n$ is the outer normal vector to $B$ and $\s$ is the surface measure on $\partial B$.
\end{lemma}
\begin{proof} We will use the Ostrogradsky-Gauss Theorem and differentiation under the integral sign. To justify these operations and make an integrand sufficiently smooth, we approximate $K(x)=x/|x|^{s+1}$ by the smooth kernel $K_\e$ in the following standard way. Let $\phi(t)$, $t\ge0$, be a $C^\infty$-function such that $\phi(t)=0$ as $0\le t\le1$, $\phi(t)=1$ as $t\ge2$, and $0\le\phi'(t)\le2$, $t>0$. Let $\phi_\e(t):=\phi(\frac t\e)$, $K_\e(x):=\phi_\e(|x|)K(x)$, and $\widetilde{R}_\e^s\mu:=K_\e\ast\mu$. We have
$$
\int_{\partial B}(\widetilde{R}_\e^s\mu\cdot\bold n)\,d\s=
\int_B\nabla\cdot\widetilde{R}_\e^s\mu(x)\,dm_d(x)=
\int_B\biggl[\int_{\R^d}\nabla\cdot\phi_\e(|y-x|)\frac{y-x}{|y-x|^{s+1}}\,d\mu(y)\biggr]dm_d(x).
$$
The inner integral is equal to
$$
\int_{|y-x|\le2\e}\nabla\cdot\phi_\e(|y-x|)\frac{y-x}{|y-x|^{s+1}}\,d\mu(y)
+\int_{|y-x|>2\e}\nabla\cdot\frac{y-x}{|y-x|^{s+1}}\,d\mu(y)=:I_1(x)+I_2(x).
$$
One can easily see that
$$
\bigg|\frac{\partial}{\partial x_i}\bigg[\phi_\e(|y-x|)\frac{y-x}{|y-x|^{s+1}}\bigg]\bigg|<
C\bigg[\frac1\e\frac1{|y-x|^{s}}+\frac1{|y-x|^{s+1}}\bigg]<\frac C{|y-x|^{s+1}},
\quad |y-x|\le2\e.
$$
Hence,
\begin{align*}
|I_1(x)|&<C\int_{|y-x|\le2\e}\frac1{|y-x|^{s+1}}\,d\mu(y)<C\int_0^{2\e}\frac1{t^{s+1}}\,d\mu(B(x,t))\\
&\approx\frac{\mu(B(x,2\e))}{(2\e)^{s+1}}+\int_0^{2\e}\frac{\mu(B(x,t))}{t^{s+2}}\,dt.
\end{align*}
Since $\mu$ has a continuous density with respect to $m_d$, we have $\mu(B(x,t))<A_{\mu,B}t^d$ as $t\le2\e<1$, $x\in B$. Taking into account that $s<d-1$, we obtain the relation $\int_BI_1(x)\,dm_d(x)\to0$ as $\e\to0$.

To estimate the integral of $I_2(x)$ we use the equality $\nabla\cdot\frac x{|x|^{s+1}}=\frac {d-s-1}{|x|^{s+1}}$. Thus,
\begin{align*}
\bigg|\int_BI_2(x)\,dm_d(x)\bigg|&=C\int_B\bigg[\int_{|y-x|>2\e}\frac{d\mu(y)}{|y-x|^{s+1}}\bigg]dm_d(x)\\
&=C\bigg(\int_{B(x_0,r+2\e)}\bigg[\int_{B\cap\{|y-x|>2\e\}}\frac{dm_d(x)}{|y-x|^{s+1}}\bigg]d\mu(y)\\
&+\int_{\R^d\setminus B(x_0,r+2\e)}\bigg[\int_B\frac{dm_d(x)}{|y-x|^{s+1}}\bigg]d\mu(y)\bigg)
=:C(J_1+J_2).
\end{align*}
Obviously,
$$
\int_B\frac{dm_d(x)}{|y-x|^{s+1}}\approx\begin{cases}\displaystyle\int_0^r\dfrac{t^{d-1}\,dt}{t^{s+1}}\approx r^{d-s-1}, &|y-x_0|\le r,\\
\dfrac{r^d}{|y-x_0|^{s+1}}, &|y-x_0|>r.\end{cases}
$$
In order to estimate $J_1$ we note that for sufficiently small $\e$,
$$
\int_{B\cap\{|y-x|>2\e\}}\frac{dm_d(x)}{|y-x|^{s+1}}\approx \int_{B}\frac{dm_d(x)}{|y-x|^{s+1}}
\approx r^{d-s-1},\quad y\in B(x_0,r+2\e).
$$
Hence, $J_1\approx r^{d-s-1}\mu(B(x_0,r+2\e))$. Moreover,
$$
J_2\approx\int_{\R^d\setminus B(x_0,r+2\e)}\dfrac{r^d}{|y-x_0|^{s+1}}\,d\mu(y)=r^d\int_{r+2\e}^\infty\frac{d\mu(B(x_0,t))}{t^{s+1}}.
$$
Passing to the limit as $\e\to0$, we get \eqref{f22}
\end{proof}

Now we are ready to prove our main result.

\begin{theorem}\label{th23} Let $\mu$ be a non-negative measure in $\R^d$ with continuous density and compact support. Let $0<s<1$. Then \eqref{f2} holds with a constant $C$ depending only on $d$ and $s$.
\end{theorem}
\begin{proof} Let us sketch the idea of proof. Let a measure $\mu$ be such that $\mu(B(y,t))\le Ct^s$, $y\in\R^d$, $t>0$. For Lipschitz continuous compactly supported functions $\f$, $\psi$, define the form $\langle R^s(\psi\mu),\f\rangle_\mu$ by the equality
$$
\langle R^s(\psi\mu),\f\rangle_\mu=\frac12\iint_{\R^d\times\R^d}\frac{y-x}{|y-x|^{s+1}}(\psi(y)\f(x)-\psi(x)\f(y))\,d\mu(y)\,d\mu(x);
$$
the double integral exists since $|\psi(y)\f(x)-\psi(x)\f(y)|\le C_{\psi,\f}|x-y|$. If we assume in addition that $\int\psi\,d\mu=0$, we may define $\langle R^s(\psi\mu),\f\rangle_\mu$ for any (not necessarily compactly supported) bounded Lipschitz continuous function $\f$ on $\R^d$; here we follow \cite{JN}. Let $\supp\psi\in B(0,R)$. For $|x|>2R$ we have
\begin{equation*}
\begin{split}
\bigg|\int_{\R^d}\frac{y-x}{|y-x|^{s+1}}\psi(y)\,d\mu(y)\bigg|&=\bigg|\int_{\R^d}\bigg[\frac{y-x}{|y-x|^{s+1}}+ \frac{x}{|x|^{s+1}}\bigg]\psi(y)\,d\mu(y)\bigg|\\
&\le\frac{C}{|x|^{s+1}}\int_{\R^d}|y\psi(y)|\,d\mu(y)=\frac{C_\psi}{|x|^{s+1}}.
\end{split}
\end{equation*}
Choose a Lipschitz continuous compactly supported function $\xi$ which is identically 1 on $B(0,2R)$. Then we may define the form $\langle R^s(\psi\mu),\f\rangle_\mu$ as
$$
\langle R^s(\psi\mu),\f\rangle_\mu =\langle R^s(\psi\mu),\xi\f\rangle_\mu + \int_{\R^d}\bigg[\int_{\R^d}\frac{y-x}{|y-x|^{s+1}}\psi(y)\,d\mu(y)\bigg](1-\xi(x))\f(x)\,d\mu(x).
$$
The repeated integral is well defined because
$$
\int_{|x|>2R}\frac{d\mu(x)}{|x|^{s+1}}\le C\int_{2R}^\infty\frac{\mu(B(0,t))}{t^{s+2}}\,dt\le C\int_{2R}^\infty\frac1{t^2}\,dt.
$$

Assuming that Theorem \ref{th23} is incorrect and using the Cotlar inequality we establish the existence of a positive measure $\nu$ such that $\nu$ has no point masses, the operator $R_\nu^s$ is bounded on $L^2(\nu)$, and $\langle R^s(\psi\nu),1\rangle_\nu=0$ for every Lipschitz continuous function $\psi$ with $\int\psi\,d\mu=0$. It means that $\nu$ is a reflectionless measure, that is a measure without point masses with the following properties: $R_\nu^s$ is bounded on $L^2(\nu)$, and $\langle R^s(\psi\nu),1\rangle_\nu=0$ for every Lipschitz continuous compactly supported function $\psi$ such that $\int\psi\,d\mu=0$. But according to the recent result by Prat and Tolsa \cite{PT} such measures do not exist for $0<s<1$. We remark that the proof of this result contains estimates of an analog of the Melnikov's curvature of a measure. This is the obstacle to extent the result to $s\ge1$. We now turn to the details.

Suppose that $C$ satisfying \eqref{f2} does not exists. Then for every $n\ge1$ there is a positive measure $\mu_n$ such that
$$
\sup_{x\in\R^d}|R^s\mu_n(x)|=1,\quad \sup_{x\in\supp\mu_n}|R^s\mu_n(x)|\le\frac1n.
$$
Let
$$
\theta_\mu(x,r):=\frac{\mu(B(x,r))}{r^s},\quad \t_\mu:=\sup_{x,r}\theta_\mu(x,r).
$$
We prove that
\begin{equation}\label{f23}
0<c<\t_{\mu_n}<C.
\end{equation}
The estimate from above is a direct consequence of Lemma \ref{le22}. Indeed, for any ball $B(x,r)$ \eqref{f22} implies the estimate
$$
c_dr^{d-1}\ge\bigg|\int_{\partial B}(R^s\mu_n\cdot\bold n)\,d\s\bigg|\ge Cr^{d-s-1}\mu_n(B),
$$
which implies the desired inequality.

The estimate from below follows immediately from a Cotlar-type inequality
$$
\sup_{x\in\R^d}|R^s\mu_n(x)|\le C\big[\sup_{x\in\supp\mu_n}|R^s\mu_n(x)|+\t_{\mu_n}\big]
$$
(see \cite[Theorem 7.1]{NTV98} for a more general result).

Let $B(x_n,r_n)$ be a ball such that $\t_{\mu_n}(x_n,r_n)>c=c(s,d)$, and let $\nu_n(\cdot)=r_n^{-s}\mu_n(x_n+r_n\cdot)$. Then
$$
R^s\mu_n(x)=R^s\nu_n\bigg(\frac{x-x_n}{r_n}\bigg), \quad \t_{\nu_n}(y,t)=\t_{\mu_n}(r_ny+x_n,r_nt).
$$
In particular, $\nu_n(B(0,1))=\t_{\mu_n}(x_n,r_n)>c$. Choosing a weakly converging subsequence of $\{\nu_n\}$, we obtain a positive measure $\nu$. If we prove that

(a) $\nu(B(y,t))\le Ct^s$,

(b) $\langle R^s\nu,\psi\rangle_\nu=0$ for every Lipschitz continuous compactly supported function $\psi$ with $\int\psi\,d\nu=0$,

(c) the operator $R_\nu^s$ is bounded on $L^2(\nu)$,

\noindent then $\nu$ is reflectionless, and we come to contradiction with Theorem~1.1 in \cite{PT} mentioned above. Thus, the proof would be completed.

The property (a) follows directly from \eqref{f23}. For weakly converging measures $\nu_n$ with $\t_{\mu_n}<C$ we may apply Lemma~8.4 in \cite{JN} which yields (b).
To establish (c) we use the inequality
$$
R^{s,\ast}\mu(x):=\sup_{\e>0}|R_{\mu,\e}^s1(x)|\le\|R^s\mu\|_{L^\infty(m_d)}+C,\quad x\in\R^d,\ C=C(s),
$$
for any positive Borel measure $\mu$ such that $\mu(B(x,r))\le r^s,\ x\in\R^d,\ r>0$, -- see \cite[Lemma~2]{Vih} or \cite[p.~47]{Vo}, \cite[Lemma~5.1]{AE} for a more general setting. Thus, $R_{\e}^s\nu_n(x):=R_{\nu_n,\e}^s1(x)\le C$ for every $\e>0$. Hence, $R_{\e}^s\nu(x)\le C$ for $\e>0$, $x\in\R^d$, and the non-homogeneous $T1$-theorem \cite{NTV03} implies the boundedness of $R_\nu^s$ on $L^2(\nu)$.
\end{proof}

\section{The case of radial density}\label{sect3}

Lemma \ref{le22} allows us to prove the maximum principle for all $s\in(0,d)$ in the special case of radial density.

\begin{proposition}\label{pr31} Let $d\mu(x)=h(|x|)\,dm_d(x)$, where $h(t)$ is a continuous function on $[0,\infty)$, and let $s\in(0,d-1)$. Then \eqref{f2} holds with a constant $C$ depending only on $d$ and $s$.
\end{proposition}
We remind that for $s\in[d-1,d)$ Conjecture \ref{conj11} is proved in \cite{ENV} for any compactly supported measure with $C^\infty$ density. Thus, for compactly supported radial measures with $C^\infty$ density \eqref{f2} holds for all $s\in(0,d)$.

\begin{proof} Because $\mu$ is radial, by \eqref{f22} we have
\begin{align*}
c_dr^{d-1}|R^s\mu(x)|&=\bigg|\int_{\partial B(0,r)}(R^s\mu\cdot\bold n)\,d\s\bigg|\\
&\approx r^{d-s-1}\mu(B(0,r))+r^d\int_r^\infty\frac{d\mu(B(0,t))}{t^{s+1}},\quad r=|x|.
\end{align*}
Thus,
$$
|R^s\mu(x)|\approx\frac{\mu(B(0,r))}{r^{s}}+r\int_r^\infty\frac{d\mu(B(0,t))}{t^{s+1}}.
$$
Fix $w\not\in\supp\mu$, and let $r=|w|$. If
$$
\frac{\mu(B(0,r))}{r^{s}}\ge r\int_r^\infty\frac{d\mu(B(0,t))}{t^{s+1}},
$$
then there is $r_1\in(0,r)$ such that $\{y:|y|=r_1\}\subset\supp\mu$ and $\mu(B(0,r))=\mu(B(0,r_1))$. Hence,
$$
|R^s\mu(w)|\approx\frac{\mu(B(0,r))}{r^s}<\frac{\mu(B(0,r_1))}{r_1^s}\le C|R^s\mu(x_1)|,\quad |x_1|=r_1.
$$
If
$$
\frac{\mu(B(0,r))}{r^{s}}< r\int_r^\infty\frac{d\mu(B(0,t))}{t^{s+1}},
$$
then there is $r_2>r$ such that $\{y:|y|=r_2\}\subset\supp\mu$ and $\mu(B(0,r))=\mu(B(0,r_2))$. Hence,
$$
\frac{\mu(B(0,r_2))}{r_2^s}<\frac{\mu(B(0,r))}{r^s}< r\int_{r_2}^\infty\frac{d\mu(B(0,t))}{t^{s+1}},
$$
and we have
$$
|R^s\mu(w)|\approx r\int_{r_2}^\infty\frac{d\mu(B(0,t))}{t^{s+1}}\le C|R^s\mu(x_2)|,\quad |x_2|=r_2.
$$
\end{proof}

\section{Counterexample}\label{sect4}

Given $\e>0$, we construct a signed measure $\nu=\nu(\e)$ in $\R^5$ with the following properties:

(a) $\nu$ is a radial signed measure with $C^\infty$-density;

(b) $\supp\nu\in D_\e:=\{1-\e\le|x|\le1+\e\}$;

(c) $|R^2\nu(x)|<\e$ for $x\in\supp\nu$; $|R^2\nu(x)|>a>0$ for $|x|=2$, where $a$ is an absolute constant. Here $R^2\nu$ means $R^s\nu$ with $s=2$.

Let $\D^2:=\D\circ\D$, and let
$$
u(x)=\begin{cases}2/3,&|x|\le1,\\
\dfrac1{|x|}-\dfrac1{3|x|^3},& |x|>1.\end{cases}
$$
Note that $\D(\frac1{|x|^3})=0$ and $\D^2(\frac1{|x|})=0$ in $\R^5\setminus\{0\}$. Hence, $\D^2u(x)=0, |x|\ne1$. Moreover, $\nabla u$ is continuous in $\R^5$ and $\nabla u(x)=0$, $|x|=1$.

For $\d\in(0,\e)$, let $\f_\d(x)$ be a $C^\infty$-function in $\R^5$ such that $\f_\d>0$, $\supp\f_\d=\{x\in\R^5:|x|\le\d\}$, and $\int\f_\d(x)\,dm_5(x)=1$ (for example, a bell-like function on $|x|\le\d\}$). Let $U_\d:=u\ast\f_\d$. Then $\D^2U_\d(x)=0$ as $x\not\in D_\d$. Also, $\D U_\d(x)\to0$ as $|x|\to\infty$. Hence, the function $\D U_\d$ can be represented in the form
$\D U_\d=c(\frac1{|x|^3}\ast\D^2U_\d)$ (here and in the sequel by $c$ we denote various absolute constants). Set $d\nu_\d=\D^2U_\d\,dm_5$. Then $\supp\nu_\d\in D_\d$, and (b) is satisfied. Since $\D(\frac1{|x|})=\frac c{|x|^3}$, we have $\D U_\d=c\D(\frac1{|x|}\ast\D^2U_\d)$, that is $U_\d=c(\frac1{|x|}\ast\D^2U_\d)+h$, where $h$ is a harmonic function in $\R^5$. Since both $U_\d$ and $\frac1{|x|}\ast\D^2U_\d$ tend to 0 as $x\to\infty$, we have
$$
U_\d=c\bigg(\frac1{|x|}\ast\D^2U_\d\bigg).
$$
Thus,
$$
R^2\nu(x)=c\int\frac{y-x}{|y-x|^3}\,d\nu_\d(y)=c\nabla U_\d(x).
$$
Obviously, $\nabla U_\d=\nabla(u\ast\f_\d)=(\nabla u)\ast\f_\d$, and hence $\max_{x\in D_\d}|\nabla U_\d(x)|\to0$ as $\d\to0$. On the other hand, for fixed $x$ with $|x|>1$ (say, $|x|=2$) we have $\lim_{\d\to0}|\nabla u\ast\f_\d|=|\nabla u(x)|>0$. Thus, (c) is satisfied if $\d$ is chosen sufficiently small.

\vspace{.2cm}
{\bf Remark.}
It is well-known that the maximum principle (with a constant $C$) holds for potentials $\int K(|x-y|)\,d\mu(y)$ with non-negative kernels $K(t)$ decreasing on $(0,\infty)$, and non-negative finite Borel measures $\mu$. Our arguments show that for non-positive measures the analog of \eqref{f2} fails even for potentials with positive Riesz kernels. In fact we have proved that {\it for every $\e>0$, there exists a signed measure $\eta=\eta(\e)$ in $\R^5$ with $C^\infty$-density and such that $\supp\eta\in D_\e:=\{1-\e\le|x|\le1+\e\}$, $|u_\eta(x)|<\e$ for $x\in\supp\eta$, but $|u_\eta((2,0,\dots,0))|>b>0$}, where $u_\eta(x):=\int\frac{d\eta(y)}{|y-x|}$, and $b$ is an absolute constant.

Indeed, for the first component $R_1^2\nu$ of $R^2\nu$ we have
$$
R_1^2\nu=c\frac{\partial}{\partial x_1}U_\d=c\bigg(\frac1{|x|}\ast\frac{\partial}{\partial x_1}(\D^2U_\d)\bigg)=cu_\eta,\text{ where } d\eta=\frac{\partial}{\partial x_1}(\D^2U_\d)\,dm_5.
$$

\end{document}